\documentclass[12pt]{amsart}
\usepackage{amsfonts,amsmath,amssymb,amsthm}
\usepackage{graphics}
\usepackage{fullpage}
\usepackage{hyperref}
\usepackage[english]{babel}
\usepackage[latin1]{inputenc}
\usepackage{cite}
\allowdisplaybreaks

\newtheorem{theorem}{Theorem}[section]
\newtheorem{lemma}[theorem]{\bf Lemma}

\newtheorem{prop}[theorem]{Proposition}
\newtheorem{definition}[theorem]{Definition}

\newtheorem{remark}[theorem]{\textbf{Remark}}

\newcommand{\N}{\mathbb{N}}
\newcommand{\Z}{\mathbb{Z}}

\newcommand{\R}{\mathbb{R}}
\newcommand{\D}{\mathcal{D}}
\newcommand{\calS}{\mathcal{S}}

\DeclareMathOperator{\supp}{supp}

\newcommand{\eM}{M_\epsilon}

\newcommand{\bs}{\backslash}
\newcommand{\bk}{\backslash}

\def\Xint#1{\mathchoice
   {\XXint\displaystyle\textstyle{#1}}%
   {\XXint\textstyle\scriptstyle{#1}}%
   {\XXint\scriptstyle\scriptscriptstyle{#1}}%
   {\XXint\scriptscriptstyle\scriptscriptstyle{#1}}%
   \!\int}
\def\XXint#1#2#3{{\setbox0=\hbox{$#1{#2#3}{\int}$}
     \vcenter{\hbox{$#2#3$}}\kern-.5\wd0}}

\def\dashint{\Xint-}

\newcommand{\pp}{{p(\cdot)}}
\newcommand{\Pp}{\mathcal{P}}
\newcommand{\Lpp}{L^\pp}
\newcommand{\cpp}{{p'(\cdot)}}
\newcommand{\Lcpp}{L^\cpp}

\begin{document}

\title{The $\epsilon$-Maximal Operator and Haar Multipliers On Variable Lebesgue Spaces}

\author{David Cruz-Uribe}

\address{Department of Mathematics, University of Alabama, Tuscaloosa, AL 35487, USA}

\author{Michael Penrod}

\address{Department of Mathematics, University of Alabama, Tuscaloosa, AL 35487, USA}

\thanks{We thank Cody Stockdale for sharing his intuition and advice on $\epsilon$-sparse operators.}



\date{\today}

\begin{abstract}
C. Stockdale, P. Villarroya, and B. Wick introduced the $\epsilon$-maximal operator to prove the Haar multiplier is bounded on the weighted spaces $L^p(w)$ for a class of weights larger than $A_p$. We prove the $\epsilon$-maximal operator and Haar multiplier are bounded on variable Lebesgue spaces $\Lpp(\R^n)$ for a larger collection of exponent functions than the log-Holder continuous functions used to prove the boundedness of the maximal operator on $\Lpp(\R^n)$. We also prove that the Haar multiplier is compact when restricted to a dyadic cube $Q_0$. 
\end{abstract}
\maketitle
\section{Introduction}
\label{sec:Introduction}
In \cite{VLS}, it was proved that the dyadic maximal operator $M^d$ is bounded on variable Lebesgue spaces $\Lpp(\R^n)$ for log-Holder continuous exponent functions $\pp\in LH(\R^n)$ with $1 < p_- \leq p_+ < \infty$. In \cite{SVW2021}, the $\epsilon$-maximal operator and $\epsilon$-sparse operator were introduced to establish boundedness for the Haar multiplier on $L^p(w)$ for a class of weights larger than $A_p$.  Motivated by these two results, we prove the $\epsilon$-maximal operator and the Haar multiplier are bounded on variable Lebesgue spaces for a collection of exponent functions larger than $LH(\R^n)$. In addition, we prove a local compactness result for the Haar multiplier similar to the result in \cite{SVW2021}. 

Before stating our results precisely, we briefly explain some of the definitions involved. These are stated in more detail in Section \ref{sec:Preliminaries}. An exponent function is a Lebesgue measurable function $\pp : \R^n \to [1,\infty)$. We denote the essential infimum and supremum of $\pp$ by $p_-$ and $p_+$. In this paper, we only consider exponent functions where $1< p_- \leq p_+ <\infty$. Given such an exponent function $\pp$, we can define the variable Lebesgue space $\Lpp(\R^n)$ as the collection of Lebesgue measurable functions satisfying $\|f\|_{\pp}<\infty$

We denote the set of all dyadic cubes in $\R^n$ by $\D$. In all our definitions and results, $\epsilon = \{\epsilon_Q\}_{Q\in \D}$ is a bounded collection of real numbers indexed by dyadic cubes $Q\in \D$ such that for any $P, Q \in \D$, if $P \subseteq Q$, then $\epsilon_P \leq \epsilon_Q$. We refer to this assumption as the domination property of $\epsilon$. Given such a collection, we define the Haar multiplier $T_\epsilon$ for all $f \in L^1_{\text{loc}}(\R^n)$ by
\begin{align}\label{def:HaarMult}
T_\epsilon f = \sum_{Q \in \D} \epsilon_Q \langle f, h_Q\rangle h_Q,
\end{align}
Here, $\langle f, h_Q\rangle = \int_Q f(y)h_Q(y) dy$, and $h_Q$ is the Haar function adapted to $Q$ defined by
\[ h_Q = |Q|^{-1/2} \left( \chi_Q - \frac{1}{2^n} \chi_{\widehat{Q}}\right),\]
where $\widehat{Q}$ is the dyadic parent of $Q$. We will prove that the Haar multiplier is bounded on $\Lpp(\R^n)$ under certain assumptions on $\pp$. To do so, we will use the $\epsilon$-maximal operator as a tool to control the Haar multiplier. The $\epsilon$-maximal operator $\eM$ is defined by
\begin{align}\label{def:EpsMaxOp}
\eM f(x) = \sup_{Q\in \D} \epsilon_Q \dashint_Q |f(y)| dy\; \chi_Q(x).
\end{align}
It is known that the dyadic maximal operator $M^d$ is bounded on $\Lpp(\R^n)$ when $\pp \in LH_0(\R^n)\cap LH_\infty(\R^n)$. The set $LH_\infty(\R^n)$ denotes the collection of exponent functions $\pp$ that are log-Holder continuous at infinity, i.e., there exists constants $C_\infty$ and $p_\infty$ such that for all $x \in \R^n$,
\begin{align}\label{LHinfty}
|p(x)-p_\infty| \leq \frac{C_\infty}{\log(e+|x|)}.
\end{align}
The set $LH_0(\R^n)$ consists of exponent functions that are locally log-Holder continuous, meaning there exists a constant $C_0$ such that for all $x,y \in \R^n$ with $|x-y|<1/2$, we have
\begin{align}\label{LH0}
|p(x) - p(y)| \leq \frac{C_0}{-\log(|x-y|)}. 
\end{align}
When $p_+<\infty$, $LH_0(\R^n)$ is equivalent to the Diening condition: there exists a constant $C$ depending on $n$ such that given any cube $Q$, 
\begin{align}\label{DieningCondition}
|Q|^{p_-(Q) - p_+(Q)} \leq C.
\end{align}

The Diening condition can be used to give a local condition on exponent functions that is adapted to dyadic operators. In the study of Hardy martingale spaces with variable exponents, the authors of \cite{JiaoMartingales} used the following condition to prove a strong-type Doob maximal inequality:
\[\mathbb{P}(A)^{p_-(A)-p_+(A)} \leq K,\]
where $K$ is a constant depending on the exponent function $\pp$, $\mathbb{P}$ is a probability measure, and $A$ is a measurable set. In \cite{WeiszDyadicBoundCondition}, a similar inequality given by
\begin{align}
|I|^{p_-(I) - p_+(I)} \leq C,\label{Weisz:ineq}
\end{align}
where $I$ is a dyadic interval in $[0,1]$, was used to prove boundedness of some other types of maximal operators for variable exponent spaces.

Note that $\eM f(x) \leq \|\epsilon\|_\infty M^d f(x)$ where $\|\epsilon\|_\infty = \sup_{Q\in \D} \epsilon_Q$.  Consequently, if $\pp \in LH_0(\R^n)\cap LH_\infty(\R^n)$, then $\eM$ is bounded on $\Lpp(\R^n)$. However, we can replace $LH_0(\R^n)$ with the weaker local condition $\epsilon LH_0(\R^n)$.
 
The set $\epsilon LH_0(\R^n)$ consists of exponent functions $\pp$ satisfying the $\epsilon$-Diening condition
\begin{align}\label{EpsLH0}
 \left( \frac{|Q|}{\epsilon_Q}\right)^{p_-(Q)-p_+(Q)} \leq C,
\end{align}
for all $Q \in \D$, where $C$ depends only on $n$ and $\pp$. Note that if $\epsilon_Q=1$ for all $Q\in \D$, we obtain the Diening condition \eqref{DieningCondition}, discussed earlier. However, the $\epsilon$-Diening condition \eqref{EpsLH0} is a dyadic condition that is also adapted to the collection $\epsilon$. This makes $\epsilon LH_0(\R^n)$ a natural condition to use to prove the $\epsilon$-maximal operator is bounded on $\Lpp(\R^n)$.
\begin{theorem}\label{EpsMaxOpBdd}
Given $\epsilon = \{\epsilon_Q\}_{Q\in \D}$, let $\pp \in LH_\infty(\R^n)\cap \epsilon LH_0(\R^n)$ with $1 < p_- \leq p_+ < \infty$. Then $\eM$ is bounded on $\Lpp(\R^n)$: i.e. there exists a constant $C=C(n,\pp, \epsilon)$ such that for any $f \in \Lpp(\R^n)$, 
\[ \|\eM f\|_{\pp} \leq C\|f\|_{\pp}.\]
\end{theorem}

Theorem \ref{EpsMaxOpBdd} can then be used as a tool to prove the Haar multiplier is bounded on $\Lpp$.
\begin{theorem}\label{HaarBound}
Given $\epsilon = \{\epsilon_Q\}_{Q\in \D}$, let $\sqrt{\epsilon} = \{\sqrt{\epsilon_Q}\}_{Q\in \D}$. If $\pp \in \sqrt{\epsilon}LH_0(\R^n)\cap LH_\infty(\R^n)$ with $1< p_- \leq p_+ < \infty$, then the Haar multiplier $T_\epsilon$ is bounded on $\Lpp(\R^n)$. 
\end{theorem}

If we restrict ourselves to $\Lpp(Q_0)$ for some dyadic cube $Q_0$, we can improve the properties of the Haar multiplier to get compactness. Let $\D(Q_0)$ be the collection of dyadic cubes contained in $Q_0$, and $\epsilon LH_0(Q_0)$ be the same as \eqref{EpsLH0}, but for $Q\in \D(Q_0)$. 
\begin{theorem}\label{HaarCompact}
Given $Q_0 \in \D$, $\epsilon = \{\epsilon_Q\}_{Q\in \D(Q_0)}$ and $0 < \alpha < 1/2$, let $\epsilon^{\alpha} = \{\epsilon_Q^\alpha\}_{Q \in \D(Q_0)}$. Suppose
\begin{align}\label{EpsDecay}
\lim_{N\to \infty} \sup\{\epsilon_Q :  \ell(Q) < 2^{-N}\} = 0
\end{align}
and $\pp \in \epsilon^{\alpha}LH_0(Q_0)$ with $1 < p_- \leq p_+ < \infty$. Then the Haar multiplier is compact on $\Lpp(Q_0)$. 
\end{theorem}
\begin{remark}
In \cite[Section 2.4]{SVW2021}, the authors give a compactness result for weighted spaces on all of $\R^n$. Their result requires a condition on the collection $\epsilon$ that implies that $\epsilon_Q \to 0$ as $\ell(Q) \to \infty$. However, this is impossible if $\epsilon$ has the domination property. Their proof, however, gives a local compactness property, and our proof is modeled on theirs. 

The compactness result for weighted spaces $L^p(w)$ on $\R^n$ was proved using extrapolation in \cite[Section 5]{ExtrapComp}, but our local compactness result on variable Legesgue spaces is new and cannot be proved using the extrapolation results in \cite{ExtrapComp}.
\end{remark}

The remainder of this paper is organized as follows. In Section \ref{sec:Preliminaries},
we state the necessary definitions and lemmas for variable Lebesgue spaces. We prove Theorem \ref{EpsMaxOpBdd} in Section \ref{sec:ProofEpsMaxOpsBdd}. In Section \ref{sec:HaarMult}, we prove Theorems \ref{HaarBound} and \ref{HaarCompact}. Lastly, in Section \ref{sec:Examples}, we show that the $\epsilon LH_0(\R^n)$ hypothesis of Theorem \ref{EpsMaxOpBdd} is weaker than the local log-Holder continuity condition defined in inequality \eqref{LH0}. We do this by showing there are exponent functions that are not in $LH_0(\R^n)$, but are in $\epsilon LH_0(\R^n)$ for some $\epsilon$. 

Throughout this paper, $C$ will denote a constant that may vary in value from line to line and which will depend on underlying parameters. If we want to specify the dependence, we will write, for instance, $C(n,\epsilon)$. If the value of the constant is not important, we will often write $A \lesssim B$ instead of $A \leq cB$ for some constant $c$. We will also use the convention that $1/\infty=0$.
%


\section{Preliminaries}
\label{sec:Preliminaries}
We begin with the necessary definitions related to variable Lebesgue spaces. We refer the reader to \cite{VLS} for more information.
\begin{definition}
An exponent function on a set $\Omega$ is a Lebesgue measurable function $\pp :\Omega \to[1,\infty)$. Denote the collection of exponent functions on $\Omega$ by $\Pp(\Omega)$. Denote the essential infimum and essential supremum of $\pp$ on a set $E$ by $p_-(E)$ and $p_+(E)$, respectively. Denote $p_+(\Omega)$ by $p_+$ and $p_-(\Omega)$ by $p_-$.
\end{definition}

\begin{definition}
Given $\pp \in \Pp(\Omega)$ with $p_+<\infty$, and a Lebesgue measurable function $f$, define the modular associated with $\pp$ by
\[ \rho_{\pp}(f) = \int_{\Omega} |f(x)|^{p(x)} dx.\]
If $f(\cdot)^{\pp} \not\in L^1(\Omega)$, define $\rho_{\pp} = +\infty$. In situations where there is no ambiguity we will simply write $\rho(f)$.
\end{definition}

\begin{definition}
Given $\pp \in \Pp(\Omega)$, define the space $\Lpp(\Omega)$ as the set of Lebesgue measurable functions $f$ satifying $\|f\|_{\Lpp(\Omega)} < \infty$, where the norm $\|\cdot\|_{\Lpp(\Omega)}$ is defined as
\[\|f\|_{\Lpp(\Omega)} = \inf\{\lambda >0 : \rho_{\pp} (f/\lambda) \leq 1\}.\]
In situations where there is no ambiguity, we will write $\|f\|_{\pp}$ instead of $\|f\|_{\Lpp(\Omega)}$. 
\end{definition}

The following propositions relate the modular and the norm and will be used to prove Theorem \ref{EpsMaxOpBdd}. The first proposition allows us to conclude a norm is finite when the modular is finite.
\begin{prop}\cite[Proposition 2.12]{VLS}\label{ModNormEquiv}
Given $\pp \in \Pp(\Omega)$ with $p_+<\infty$, $f\in \Lpp(\Omega)$ if and only if $\rho(f)<\infty$.
\end{prop}
\begin{prop}\cite[Corollary 2.22]{VLS}\label{ModNorm:ineq1}
Let $\pp \in \Pp(\Omega)$. If $\|f\|_{\pp} \leq 1$, then $\rho(f) \leq \|f\|_{\pp}$. 
\end{prop}

In \cite[Theorem 3.16]{VLS}, it is proven that the Hardy-Littlewood maximal operator $M$ is bounded on $\Lpp(\R^n)$ when $\pp \in LH_0(\R^n)\cap LH_\infty(\R^n)$ with $p_+<\infty$. Since the dyadic maximal operator $M^d$ is bounded pointwise by the Hardy-Littlewood maximal operator, i.e. $M^d f(x) \leq Mf(x)$, the same assumptions on $\pp$ ensure that $M^d$ is bounded on $\Lpp(\R^n)$. We can weaken the local assumption $LH_0(\R^n)$ by incorporating the collection $\epsilon$ into the definition of $\epsilon LH_0(\R^n)$ stated in inequality \eqref{EpsLH0} from Section \ref{sec:Introduction}.

Unfortunately, we cannot incorporate the collection $\epsilon$ into the $LH_\infty(\R^n)$ definition in any way to weaken it. This is due to the fact that the $\epsilon$-maximal operator is pointwise equivalent to the dyadic maximal operator near infinity. More precisely, given a function $f$ that is bounded and supported on a dyadic cube $Q_0$, we have that
\[\epsilon_{Q_0} M^df(x) \leq \eM f(x) \leq \|\epsilon\|_\infty M^d f(x),\]
for almost every $x  \not\in Q_0$. Since the constants $\epsilon_{Q_0}$ and $\|\epsilon\|_\infty$ do not depend on any information about $\epsilon_Q$ for $Q\neq Q_0$, any condition near infinity that we use to bound $\eM f$ outside of $Q_0$ will have to be the same condition we use to bound $M^d f$ outside of $Q_0$, and not a condition based on the properties of the collection $\epsilon$. 

Since we are assuming that $\pp \in LH_\infty(\R^n)$, we can use following lemma when proving Theorem \ref{EpsMaxOpBdd}.
\begin{lemma}\cite[Lemma 3.26]{VLS}\label{LHinftyLemma}
Let $\pp \in LH_\infty(\R^n)$ with $1<p_-\leq p_+ <\infty$. Let $R(x)=(e+|x|)^{-n}$. Then there exists a constant $C$, depending on $n$ and the $LH_\infty$ constants of $\pp$, such that given any set $E$ and any function $F$ with $0\leq F(x) \leq 1$, for $x\in E$, 
\begin{align}
\int_E F(x)^{p(x)} dx &\leq C \int_E F(x)^{p_\infty} dx + \int_E R(x)^{p_-} dx, \label{LHinftyLemma:ineq1}\\
\int_E F(x)^{p_\infty} dx & \leq C \int_E F(x)^{p(x)}dx + \int_E R(x)^{p_-} dx. \label{LHinftyLemma:ineq2}
\end{align}
\end{lemma}

We now provide the definition and properties of dyadic cubes. These properties are well-known and can be found in \cite[Section 3.2]{VLS}.
\begin{definition}
Let $Q_0=[0,1)^n$, and let $\D_0$ be the set of all translates of $Q_0$ whose vertices are on the lattice $\Z^n$. More generally, for each $k\in \Z$, let $Q_k = 2^{-k}Q_0=[0,2^{-k})^n$, and let $\D_k $ be the set of all translates of $Q_k$ whose vertices are on the lattice $2^{-k} \Z^n$. Define the set of dyadic cubes $\D$ by 
\[ \D = \bigcup_{k\in \Z} \D_k.\]
\end{definition}
\begin{prop}\label{DyadicCubeProperties}
Dyadic cubes have the following properties:
\begin{enumerate}
\item For each $k \in \Z$, if $Q \in \D_k$, then $\ell(Q)=2^{-k}$, where $\ell(Q)$ is the side length of $Q$.

\item For each $x \in \R^n$ and $k \in \Z$, there exists a unique cube $Q \in \D_k$ such that $x\in Q$.

\item Given any two cubes $Q_1, Q_2 \in \D$, either $Q_1\cap Q_2 = \emptyset$, $Q_1 \subset Q_2$, or $Q_2 \subset Q_1$.

\item For each $k \in \Z$, if $Q\in \D_k$, then there exists a unique cube $\widehat{Q} \in \D_{k-1}$ such that $Q\subset \widehat{Q}$. ($\widehat{Q}$ is referred to as the dyadic parent of $Q$.)

\item For each $k \in \Z$, if $Q \in \D_k$, then there exist $2^n$ cubes $P_i \in \D_{k+1}$ such that $P_i \subset Q$.
\end{enumerate}
\end{prop}

The following proposition presents an equivalent characterization of $\epsilon LH_0(\R^n)$ which will be used in the proof of Theorem \ref{EpsMaxOpBdd}.
\begin{prop}\label{EpsLH0Penrod}
Given $\epsilon=\{\epsilon_Q\}_{Q\in\D}$, $\pp \in \epsilon LH_0(\R^n)$ if and only if there exists $C >0$ such that for all $Q\in \D$ and $x\in Q$,
\begin{align}
\left(\frac{|Q|}{\epsilon_Q}\right)^{p_-(Q)-p(x)} \leq C. \label{EpsLH0Penrod:ineq}
\end{align}
\end{prop}
\begin{proof}
Assume $\pp \in \epsilon LH_0(\R^n)$. Fix $Q \in \D$. Observe that if $|Q|/\epsilon_Q >1$, then for any $x \in Q$, \eqref{EpsLH0Penrod:ineq} holds with $C=1$. Suppose $|Q|/\epsilon_Q\leq 1$. Then for any $x\in Q$, we have
\[\left( \frac{|Q|}{\epsilon_Q}\right)^{p_-(Q)-p(x)} \leq \left( \frac{|Q|}{\epsilon_Q}\right)^{p_-(Q)-p_+(Q)} \leq C.\]
To prove the reverse direction, observe that if $|Q|/\epsilon_Q >1$, then \eqref{EpsLH0Penrod:ineq} holds with $C=1$. Suppose $|Q|/\epsilon_Q \leq 1$. Let $\delta >0$ be arbitrarily small and choose $x_0 \in Q$ such that $p(x_0) + \delta > p_+(Q)$. Then by the definition of $\epsilon LH_0(\R^n)$, we have
\[\left( \frac{|Q|}{\epsilon_Q}\right)^{p_-(Q) -p_+(Q)} \leq \left( \frac{|Q|}{\epsilon_Q}\right)^{p_-(Q) - p(x_0)-\delta} \leq C \left( \frac{|Q|}{\epsilon_Q}\right)^{-\delta}.\]
Letting $\delta$ tend to $0$, we see that $\pp \in \epsilon LH_0(\R^n)$.
\end{proof}

In order to prove Theorem \ref{EpsMaxOpBdd}, we need the following Calderon-Zygmund decomposition for the $\epsilon$-maximal operator. This is very similar to the classical Calderon-Zygmund decomposition for the dyadic maximal operator \cite[Lemma 3.9]{VLS}. For the convenience of the reader we include the short proof.
\begin{lemma}\label{CZ}
Let $f\in L^1_{loc}(\R^n)$ be such that $\dashint_{Q} |f(y)|dy \to 0$ as $|Q|\to \infty$. Given $\lambda >0$, there exists a (possibly empty) collection of disjoint dyadic cubes $\{ Q_j^\lambda\}_j$ such that
\begin{align}\label{CZdecomp}
\Omega_\lambda = \{x \in \R^n : \eM f(x) > \lambda\} = \bigcup_{j} Q_j^\lambda,
\end{align}
and for each $Q_j^\lambda$, 
\begin{align}\label{CZbounds}
\lambda < \epsilon_{Q_j^\lambda} \dashint_{Q_j^\lambda} |f(y)| dy \leq 2^n \lambda.
\end{align}
\end{lemma}
\begin{proof}
If $\Omega_\lambda$ is empty, then we choose an empty collection and the conclusions hold trivially. Suppose $\Omega_\lambda$ is nonempty and let $x \in \Omega_\lambda$. Then there exists $Q \in \D$ containing $x$ such that 
\[ \epsilon_{Q} \dashint_{Q} |f(y)| dy > \lambda.\]

Since $\{\epsilon_Q\}_{Q\in \D}$ is bounded and $\dashint_{Q}|f(y)| dy \to 0$ as $|Q|\to \infty$, there is a maximal dyadic cube with this property. Denote it by $Q_x$. Clearly, $\Omega_\lambda \subseteq \bigcup_{x\in \Omega_\lambda} Q_x$. The reverse inclusion holds as well. To see this, consider any $Q_x$ and let $z \in Q_x$. Then 
\[\eM f(z) \geq \epsilon_{Q_x} \dashint_{Q_x} |f(y)| dy \;\chi_{Q_x}(z) > \lambda,\]
and so $z \in \Omega_{\lambda}$. By the nature of dyadic cubes, the cubes $\{Q_x\}_{x\in \Omega_\lambda}$ are all equal or disjoint. Also, note that since each $\D_k$ is a countable collection, and $\D$ is the countable union of all $\D_k$, we have that $\D$ is countable. Consequently, there are at most countably many such cubes $Q_x$. Enumerate these cubes by $\{Q_j^\lambda\}_j$. Clearly these cubes satisfy \eqref{CZdecomp}. 

The lower bound in \eqref{CZbounds} is immediate by our choice of $\{Q_j^\lambda\}_j$. To show the upper bound, observe that by the domination property of the collection $\epsilon$, we have $\epsilon_{\widehat{Q_j^\lambda}} \geq \epsilon_{Q_j^\lambda}$. Combining this with the maximality of our choice of $Q_j^k$, 
\[ \lambda\geq \epsilon_{\widehat{Q_j^\lambda}} \dashint_{\widehat{Q_j^\lambda}} |f(y)| dy \geq \epsilon_{Q_j^\lambda} \dashint_{\widehat{Q_j^\lambda}} |f(y)| dy \geq 2^{-n} \epsilon_{Q_j^\lambda} \dashint_{Q_j^\lambda} |f(y)| dy ,\]
Multiplying by $2^n$ gives the desired upper bound.
\end{proof}
In order to prove Theorem \ref{HaarCompact}, we need a local version of Lemma \ref{CZ}. We state the local version and briefly outline the adaptations to the proof of Lemma \ref{CZ} needed to prove it. 
\begin{lemma}\label{LocalCZ}
Given $Q_0 \in \D$, $\epsilon = \{\epsilon_Q\}_{Q\in \D(Q_0)}$, and $f \in L^1_{loc}(Q_0)$, for any $\lambda > \epsilon_{Q_0} \dashint_Q |f(y)|dy$, there exists a (possibly empty) collection of disjoint cubes $\{Q_j^\lambda\}_j$ such that 
\begin{align}\label{LocalCZDecomp}
\Omega_\lambda = \{x \in Q_0 : M_\epsilon f(x) >\lambda\} = \bigcup_j Q_j^\lambda,
\end{align}
and for each $Q_j^\lambda$, 
\begin{align}\label{LocalCZbounds}
\lambda < \epsilon_{Q_j^\lambda} \dashint_{Q_j^\lambda}|f(y)| dy \leq 2^n \lambda.
\end{align}
\end{lemma}
\begin{proof}
Choose the collection $\{Q_j^\lambda\}_j$ as in the proof of Lemma \ref{CZ}. The lower bound in inequality \eqref{LocalCZbounds} is immediate. The proof of the upper bound depends on every $Q_j^\lambda$ having a dyadic parent $\widehat{Q_j^\lambda}$ in $Q_0$, which will hold if and only if $Q_0$ is not in the collection $\{Q_j^\lambda\}_j$. Recall that we chose the cubes $Q_j^\lambda$ as the maximal cubes satisfying $\epsilon_{Q_j^\lambda} \dashint_{Q_j^\lambda} |f(y)|dy > \lambda$. Since we only consider $\lambda > \epsilon_{Q_0}\dashint_{Q_0} |f(y)| dy$, we have that $Q_0$ is not in $\{Q_j^\lambda\}_j$. Hence, every cube in $\{Q_j^\lambda\}_j$ has a dyadic parent in $Q_0$, and so the proof of the upper bound in inequality \eqref{LocalCZbounds} is the same as in the proof of Lemma \ref{CZ}.
\end{proof}

The following lemma allows us to apply the Calderon-Zygmund decomposition to any function in $\Lpp(\R^n)$ when $p_+<\infty$.
\begin{lemma}\cite[Lemma 3.29]{VLS}\label{AvgDecay}
Given $\pp \in \Pp(\R^n)$, suppose $p_+<\infty$. Then for all $f\in \Lpp(\R^n)$, $\dashint_Q |f(y)|dy \to 0$ as $|Q|\to \infty$. 
\end{lemma}
In order to prove Theorem \ref{HaarBound}, we will need some lemmas about the conjugate exponent function $\cpp$ defined pointwise by
\[\frac{1}{p'(x)} = 1-\frac{1}{p(x)}.\]
The first two lemmas will allow us to transfer properties of $\pp$ to $\cpp$. The first lemma is well-known and is an immediate consequence of the definition. See \cite{VLS}.
\begin{lemma}\label{ConjExpLHinfty}
Let $\pp \in \Pp(\R^n)$ with $1 <p_- \leq p_+<\infty$. Then $\pp \in LH_\infty(\R^n)$ if and only if $\cpp\in LH_\infty(\R^n)$.
\end{lemma}
\begin{lemma}\label{EpsLH0ConjExpEquiv}
$\pp \in \epsilon LH_0(\R^n)$ if and only if $\cpp \in \epsilon LH_0(\R^n)$. 
\end{lemma}
\begin{proof}
Assume $\pp \in \epsilon LH_0(\R^n)$. Let $Q \in \D$. Observe that since $(p')_-(Q) - (p')_+(Q) \leq 0$, we have that if $|Q|/\epsilon_Q >1$, then inequality \eqref{EpsLH0} holds with $C=1$. Suppose that $(|Q|/\epsilon_Q)\leq 1$. To show $\cpp \in \epsilon LH_0(\R^n)$, it suffices to show that there is a constant $C_1>0$ depending only on $\pp$ such that for any $Q \in \D$, we have
\begin{align}
(p')_+(Q) - (p')_-(Q) \geq C_1 (p_+(Q)-p_-(Q)).\label{ConjExp:ineq1}
\end{align}
To see why, observe that since $|Q|/\epsilon_Q \leq 1$ and $\pp \in \epsilon LH_0(\R^n)$, if \eqref{ConjExp:ineq1} holds, then we have
\[ \left( \frac{\epsilon_Q}{|Q|}\right)^{(p')_+(Q)-(p')_-(Q)} \geq \left( \frac{\epsilon_Q}{|Q|}\right)^{C_1(p_+(Q)-p_-(Q))}.\]
Flipping both sides, we have that
\[\left( \frac{|Q|}{\epsilon_Q}\right)^{(p')_-(Q) - (p')_+(Q))} \leq \left( \frac{|Q|}{\epsilon_Q}\right)^{C_1(p_-(Q)-p_+(Q))}.\]
Since $\pp \in \epsilon LH_0(\R^n)$, the right hand side is bounded by a constant depending only on $n$ and $\pp$. Hence, $\cpp \in \epsilon LH_0(\R^n)$.

We now prove that inequality \eqref{ConjExp:ineq1} holds. First, recall that by the definition of conjugate exponent functions, we have
\[\frac{1}{(p')_+(Q)} = 1-\frac{1}{p_-(Q)}, \hskip.5cm \text{ and } \hskip.5cm \frac{1}{(p')_-(Q)} = 1 - \frac{1}{p_+(Q)}.\]
Using these properties, we have that
\begin{align*}
(p')_+(Q) - (p')_-(Q) & = (p')_+(Q)(p')_-(Q)\left[ \frac{1}{(p')_-(Q)} - \frac{1}{(p')_+(Q)}\right]\\
	& = (p')_+(Q)(p')_-(Q) \left[ \frac{1}{p_-(Q)} - \frac{1}{p_+(Q)}\right]\\
	& = \frac{(p')_+(Q)(p')_-(Q)}{p_-(Q)p_+(Q)}[ p_+(Q) - p_-(Q)]\\
	& \geq \frac{((p')_-)^2}{(p_+)^2}[p_+(Q) - p_-(Q)]
\end{align*}
This proves inequality \eqref{ConjExp:ineq1}, and so $\cpp \in \epsilon LH_0(\R^n)$. The proof of the converse is the same, except we interchange the roles of $\pp$ and $\cpp$.
\end{proof}

The following lemma allows us to apply the previous two lemmas in our estimates when proving Theorem \ref{HaarBound} and Theorem \ref{HaarCompact}.
\begin{lemma}\cite[Proposition 2.37]{VLS}\label{AssocNormEquiv}
Given $\pp \in \Pp(\Omega)$ with $1 <p_-\leq p_+ <\infty$, define the associate norm $\|\cdot \|_\pp'$ by
\[\|f\|_\pp' = \sup\left\{\int_{\Omega} f(x)g(x)dx : g \in \Lcpp(\Omega), \|g\|_{\cpp}\leq 1\right\}.\]
Then for any $f\in \Lpp(\Omega)$, we have
\[\|f\|_\pp \leq  \|f\|_\pp'.\]
\end{lemma}

The next lemma is the variable exponent version of Holder's inequality. 
\begin{lemma}\cite[Theorem 2.26]{VLS}\label{HolderIneq}
Given $\pp \in \Pp(\Omega)$ with $1< p_- \leq p_+ < \infty$, for all $f \in \Lpp(\Omega)$ and $g\in \Lcpp(\Omega)$, $fg\in L^1(\Omega)$ and 
\[\int_{\Omega} |f(x) g(x) |dx \leq 2 \|f\|_\pp \|g\|_\cpp.\]
\end{lemma}


\section{Boundedness of the $\epsilon$-Maximal Operator}
\label{sec:ProofEpsMaxOpsBdd}
The proof of Theorem \ref{EpsMaxOpBdd} is adapted from \cite[Theorem 3.16]{VLS}. We begin the proof by making some reductions. We may assume $f$ is nonegative since $\eM (f) = \eM(|f|)$. By homogeneity, we may further assume that $\|f\|_\pp = 1$. From Proposition \ref{ModNorm:ineq1}, we get that $\rho (f) \leq 1$. Decompose $f$ as $f_1 + f_2$, where
\[f_1 = f\chi_{\{x : f(x) >1\}}, \hskip1.5cm f_2 = f\chi_{\{x : f(x) \leq 1\}}.\]
Then $\rho(f_i) \leq \|f_i\|_\pp \leq 1$ for $i=1,2$. Further, since $\eM f \leq \eM f_1 + \eM f_2$, it will suffice to show for $i =1,2$ that $\|\eM f_i\|_{\pp} \leq C(n, \pp, \epsilon)$. Since $p_+<\infty$, by Proposition \ref{ModNormEquiv} it will in turn suffice to show that for $i=1,2$,
\[ \rho(\eM f_i) = \int_{\R^n} \eM f_i (x)^{p(x)} dx \leq C(n, \pp,\epsilon).\]

First we consider the estimate for $f_1$. Let $A=2^n$. For each $k \in \Z$, define
\[\Omega_k = \{x\in \R^n : \eM f_1(x) >A^k\}.\]

Observe that, up to a set of measure zero, $\R^n = \bigcup_{k\in \Z} \Omega_k \backslash \Omega_{k+1}$. Since $p_+<\infty$, by Lemma \ref{AvgDecay}, $f$ satisfies the hypotheses of Lemma \ref{CZ}. Thus, for each $k$ we may form a collection of pairwise disjoint cubes $\{Q_j^k\}_j$ such that \eqref{CZdecomp} and \eqref{CZbounds} hold. For each $k$, define the sets $E_j^k = Q_j^k \cap (\Omega_k \backslash \Omega_{k+1})$. Then for each $k$, $\{E_j^k\}_j$ forms a pairwise disjoint collection such that $\Omega_k\bs \Omega_{k+1} = \bigcup_j E_j^k$. 

Now observe that
\begin{align*}
\rho(\eM f_1) &= \sum_{k} \int_{\Omega_k \bs \Omega_{k+1}} \eM f_1(x)^{p(x)} dx\\
	& \leq \sum_{k} \int_{\Omega_k \bs\Omega_{k+1}} (A^{k+1})^{p(x)} dx\\
	& \leq A^{p_+} \sum_{k,j} \int_{E_j^k} \left( \epsilon_{Q_j^k} \dashint_{Q_j^k} f_1(y) dy\right)^{p(x)} dx.
\end{align*}

For each $k$ and $j$, define $p_{jk} = p_-(Q_j^k)$. Since for any $x \in \R^n$, $f_1(x) >1$ or $f_1(x)=0$, we then have
\begin{align}\label{f1Case:ineq1}
 \int_{Q_j^k} f_1(y)dy \leq \int_{Q_j^k}f_1(y)^{p(y)/p_{jk}}dy \leq \int_{Q_j^k} f_1(y)^{p(y)} dy \leq 1.
\end{align}
Now observe that by Proposition \ref{EpsLH0Penrod}, inequality \eqref{f1Case:ineq1}, and Holder's inequality we have
\begin{align*}
\sum_{k,j} \int_{E_j^k}\left( \epsilon_{Q_j^k} \dashint_{Q_j^k} f_1(y) dy\right)^{p(x)}dx= & \sum_{k,j} \int_{E_j^k} \left( \frac{\epsilon_{Q_j^k}}{|Q_j^k|}\right)^{p(x)} \left( \int_{Q_j^k} f_1(y) dy\right)^{p(x)} dx\\
	\lesssim & \sum_{k,j} \int_{E_j^k} \left( \frac{\epsilon_{Q_j^k}}{|Q_j^k|}\right)^{p_jk} \left( \int_{Q_j^k} f_1(y) dy\right)^{p(x)} dx\\
	\lesssim & (1+\|\epsilon\|_\infty)^{p_{+}} \sum_{k,j} \int_{E_j^k} |Q_j^k|^{-p_{jk}} \left( \int_{Q_j^k} f_1(y) dy\right)^{p(x)} dx\\
	\lesssim &(1+\|\epsilon\|_\infty)^{p_+} \sum_{k,j} \int_{E_j^k} |Q_j^k|^{-p_{jk}} \left( \int_{Q_j^k} f_1(y)^{p(y)/p_{jk}} dy\right)^{p(x)} dx\\
	\lesssim &(1+\|\epsilon\|_\infty)^{p_+} \sum_{k,j} \int_{E_j^k} \left( |Q_j^k|^{-1} \int_{Q_j^k} f_1(y)^{p(y)/p_{jk}}\right)^{p_{jk}} dx\\
	\lesssim &(1+\|\epsilon\|_\infty)^{p_+} \sum_{k,j} \int_{E_j^k}  \left( \dashint_{Q_j^k} f_1(y)^{p(y)/p_-}dy\right)^{p_-} dx\\
	\lesssim &(1+ \|\epsilon\|_\infty)^{p_+} \sum_{k,j} \int_{E_j^k} M^d[ f_1^{\pp/p_-}](x)^{p_-} dx\\
	= & C(\pp, \epsilon) \int_{\R^n} M^d[f_1^{\pp/p_-}](x)^{p_-} dx.\\
\end{align*}
Since $p_->1$, we have $\|M^d f_1\|_{L^{p_-}(\R^n)} \leq (p_-)' \|f_1\|_{L^{p_-}(\R^n)}$ (see \cite[Theorem 2.3]{MoenSharpBound}, \cite[Exercise 2.1.12]{GrafakosBook}). Combining this with the fact that $\rho(f_1) \leq 1$, we have that 
\[\rho(\eM f_1) \leq C(n,\pp, \epsilon) \rho(f_1) \leq C(n,\pp, \epsilon).\]

To estimate $\rho(\eM f_2)$ observe that since $f_2 \leq 1$, we have $\dashint_Q f_2(y) dy \leq 1$ for all $Q \in \D$. Thus, 
\[ \frac{\epsilon_Q}{\|\epsilon\|_{\infty}} \dashint_Q f_2(y) dy \chi_Q(x) \leq 1,\]
for all $x\in \R^n$. Hence, $0 \leq\|\epsilon\|_{\infty}^{-1} \eM f_2 \leq 1$. Let $R(x) = (e+|x|)^{-n}$. Note that since $p_->1$, we have $p_\infty >1$, and so $\int_{\R^n}M^d f_2(x)^{p_\infty}dx \leq( (p_\infty)')^{p_\infty}\int_{\R^n} f(x)^{p_\infty} dx$. Combining this with inequalities \eqref{LHinftyLemma:ineq1}, \eqref{LHinftyLemma:ineq2}, and the pointwise bound $\eM f_2(x) \leq \|\epsilon\|_\infty M^d f_2(x)$, we have that
\begin{align*}
\int_{\R^n} \eM f_2(x)^{p(x)} dx & \leq \|\epsilon\|_{\infty}^{p_+} \int_{\R^n} [\|\epsilon\|_{\infty}^{-1} \eM f_2(x)]^{p(x)} dx \\
	& \leq C\|\epsilon\|_{\infty}^{p_+} \int_{\R^n} [\|\epsilon\|_{\infty}^{-1} \eM f_2(x)]^{p_\infty} dx + \|\epsilon\|_{\infty}^{p_+} \int_{\R^n} R(x)^{p_-} dx\\
	& = C \|\epsilon\|_{\infty}^{p_+ - p_\infty} \int_{\R^n} \eM f_2(x)^{p_\infty} dx + \|\epsilon\|_{\infty}^{p_+} \int_{\R^n} R(x)^{p_-} dx\\
	& \leq C\|\epsilon\|_{\infty}^{p_+-p_\infty} \int_{\R^n} \|\epsilon\|_{\infty}^{p_\infty} M^df_2(x)^{p_\infty} + \|\epsilon\|_{\infty}^{p_+} \int_{\R^n} R(x)^{p_-} dx\\
	& \leq C ((p_\infty)')^{p_\infty} \|\epsilon\|_{\infty}^{p_+} \int_{\R^n} f_2(x)^{p_\infty} dx + \|\epsilon\|_{\infty}^{p_+} \int_{\R^n} R(x)^{p_-} dx\\
	& \leq C(n, \pp, \epsilon) \int_{\R^n} f_2(x)^{p(x)} dx + C(n,\pp,\epsilon)\int_{\R^n} R(x)^{p_-} dx.
\end{align*}

Since $\rho(f_2) \leq 1$ and $\int_{\R^n} R(x)^{p_-}$ is finite, we have that $\int_{\R^n} \eM f_2(x)^{p(x)}\leq C(n,\pp, \epsilon)$. This completes the proof of Theorem \ref{EpsMaxOpBdd}.
\bigskip

We will need a local version of Theorem \ref{EpsMaxOpBdd} to prove Theorem \ref{HaarCompact}. We state the local version and outline the modifications to the proof. Note that the necessary lemmas and propositions used to prove Theorem \ref{EpsMaxOpBdd} still hold when replacing $\R^n$ with $Q_0$. 
\begin{lemma}\label{LocalEpsMaxOpBdd}
Given $Q_0 \in \D$, $\epsilon =  \{\epsilon_Q\}_{Q\in \D(Q_0)}$, if $\pp \in \epsilon LH_0(Q_0)$ with $1 < p_- \leq p_+ < \infty$, then there exists a constant $C= C(n,\pp, \epsilon, Q_0)$ such that
\[\|\eM f\|_{\pp} \leq C\|f\|_{\pp},\]
for all $f \in \Lpp(Q_0)$.
\end{lemma}
\begin{proof}
Using the same reductions as in the proof of Theorem \ref{EpsMaxOpBdd}, we must show that $\rho(\eM f_i) \leq C$ for $i=1,2$. Since $|Q_0|$ is finite and $\eM f_2 \leq 1$, we immediately have that 
\[\rho(\eM f_2) \leq |Q_0|.\]

To estimate $\rho(\eM f_1)$, let $A=2^n$ and $k_0$ be the smallest integer such that 
\[2^{k_0} > \epsilon_{Q_0} \dashint_{Q_0} |f_1(y)| dy.\]
For $k \geq k_0$, define $\Omega_k$ the same as in Theorem \ref{EpsMaxOpBdd}, and define $\Omega = Q_0 \bk \Omega_{k_0}$. Then $Q_0 = \Omega \cup \left( \bigcup_{k \geq k_0} \Omega_k \bk \Omega_{k+1}\right)$. We now write $\rho(\eM f_1)$ as
\[ \rho(\eM f_1) = \int_{\Omega} \eM f_1(x)^{p(x)} dx + \sum_{k\geq k_0} \int_{\Omega_k \bk \Omega_{k+1}} \eM f_1(x)^{p(x)} dx.\]
The second term is estimated using $\epsilon LH_0(Q_0)$, inequality \eqref{f1Case:ineq1}, and Holder's inequality the same way as in Theorem \ref{EpsMaxOpBdd}. To estimate the first term, observe that for $x \in \Omega$, we have $\eM f_1(x) \leq A^{k_0}$. Thus 
\[\int_\Omega \eM f_1(x)^{p(x)} dx \leq A^{k_0 p_+} |Q_0|.\]
This completes the proof.
\end{proof}


\section{Haar Multipliers}
\label{sec:HaarMult}

To prove the Haar multiplier defined in \eqref{def:HaarMult} is bounded on $L^p(w)$, in \cite{SVW2021} they proved it was dominated by a sparse operator. To state their result, we need two definitions. 

\begin{definition}\label{SparseCollectionDef}
We say $\calS\subset \D$ is a sparse collection of dyadic cubes if for every $Q \in \calS$,
\[ \sum_{P \in \eta_\calS(Q)} |P| \leq \frac{1}{2}|Q|,\]
where $\eta_\calS (Q)$ is the set of maximal elements of $\calS$ that are strictly contained in $Q$. 
\end{definition}
\begin{remark}\label{SparsenessEquiv}
Given a collection of cubes $\calS$, for each $Q\in \calS$, define $E_Q = \left(\bigcup_{P \in \eta_\calS(Q)} P\right)^c\cap Q$. Then Definition \ref{SparseCollectionDef} is equivalent to the condition that $|Q| \leq 2|E_Q|$, for all $Q \in \calS$. Furthermore, $\{E_Q\}_{Q\in \calS}$ is a pairwise disjoint collection. 
\end{remark}

\begin{definition}\label{SparseOp}
Given a sparse collection $\calS$ and $\epsilon= \{\epsilon_Q\}_{Q\in \D}$, for all $f\in L^1_{\text{loc}}(\R^n)$ and $x\in \R^n$, define the $\epsilon$-sparse operator $S_\epsilon$ by
\[ S_\epsilon f(x) = \sum_{Q \in \calS} \epsilon_Q \dashint_Q f(y)dy \; \chi_Q(x),\]
\end{definition}

The following theorem allows us to reduce the proofs about $T_\epsilon$ to proofs about $\epsilon$-sparse operators.
\begin{theorem}\cite[Theorem 1.2]{SVW2021}\label{PointwiseSparseBound}
Given $\epsilon = \{\epsilon_Q\}_{Q\in \D}$, if $f$ is bounded with compact support, then there exists a sparse collection $\calS$ such that the associated $\epsilon$-sparse operator $S_\epsilon$ satisfies
\[|T_\epsilon f(x)| \lesssim S_\epsilon |f|(x)\]
for almost every $x\in \supp (f)$.
\end{theorem}

We now prove Theorem \ref{HaarBound}.
\begin{proof}
We will first prove that $\|T_\epsilon f\|_\pp \lesssim \|f\|_\pp$ for any $f\in L^\infty_c(\R^n)$. Fix such an $f$. We need only show that for the sparse collection $\calS$ from Theorem \ref{PointwiseSparseBound}, the associated $\epsilon$-sparse operator $S_\epsilon$ satisfies 
\begin{align}\label{EpsSparseOpBdd}
\|S_\epsilon f\|_\pp \lesssim \|f\|_\pp.
\end{align}

By Lemma \ref{AssocNormEquiv}, there exists $g\in \Lcpp(\R^n)$ with $\|g\|_\cpp \leq 1$ such that
\[\|S_\epsilon f\|_\pp \leq \|S_\epsilon f\|_\pp' \leq 2 \int_{\R^n} S_\epsilon f(x) g(x) dx.\]
By Remark \ref{SparsenessEquiv} and Lemma \ref{HolderIneq}, we have that
\begin{align*}
 \int_{\R^n} S_\epsilon f(x) g(x) dx	& = \sum_{Q\in \calS} \epsilon_Q\dashint_Q f(y) dy \int_Q g(x) dx\\
	& \leq 2\sum_{Q\in \calS} \sqrt{\epsilon_Q }\dashint_Q f(y) dy \sqrt{\epsilon_Q}\dashint_Q g(x) dx |E_Q| \\
	& \leq 2 \sum_{Q\in \calS} \int_{E_Q} M_{\sqrt{\epsilon}}f(t) M_{\sqrt{\epsilon}}g(t)dt\\
	& \leq 2 \int_{\R^n}  M_{\sqrt{\epsilon}}f(t) M_{\sqrt{\epsilon}}g(t) dt \\
	& \leq 4 \|M_{\sqrt{\epsilon}}f\|_{\pp} \|M_{\sqrt{\epsilon}}g\|_{\cpp}.\\
\end{align*}

Since $\pp \in \sqrt{\epsilon}LH_0(\R^n)$, by Lemma \ref{EpsLH0ConjExpEquiv} we have $\cpp \in \sqrt{\epsilon}LH_0(\R^n)$. By Lemma \ref{ConjExpLHinfty}, since $\pp \in LH_\infty(\R^n)$, we have $\cpp \in LH_\infty(\R^n)$. Hence, by Theorem \ref{EpsMaxOpBdd}, we have 
\[\|M_{\sqrt{\epsilon}}f\|_{\pp} \|M_{\sqrt{\epsilon}}g\|_{\cpp} \leq C \|f\|_\pp \|g\|_\cpp \leq C\|f\|_{\pp}.\]
Hence, $\|T_\epsilon f\|_\pp \lesssim \|f\|_\pp$ for $f\in L^\infty_c(\R^n)$. 

Now given any $f\in \Lpp(\R^n)$, since $L^\infty_c(\R^n)$ is dense in $\Lpp(\R^n)$ (See \cite[Theorem 2.72]{VLS}) and $T_\epsilon$ is linear, the desired inequality follows by a standard approximation argument.
\end{proof}

We now prove Theorem \ref{HaarCompact}. 
\begin{proof}
It suffices to show that for any sparse collection $\calS \subset \D(Q_0)$, the associated $\epsilon$-sparse operator $S_\epsilon$ is compact on $\Lpp(Q_0)$. Fix $\calS \subset \D(Q_0)$. For each $N \in \N$, define the set $D_N$ by
\[D_N = \{Q\in \D(Q_0) : 2^{-N} \leq \ell(Q) \leq 2^N\},\]
and define the operator $S_{\epsilon,N}$ by 
\[S_{\epsilon,N} f = \sum_{Q\in D_N \cap \calS} \epsilon_Q \dashint_Q f(y) dy \; \chi_Q,\]
for $f\in L^\pp(Q_0)$. Since $Q_0$ is bounded, $D_N$ is a finite collection for all $N$. Hence $S_{\epsilon,N}$ is a finite rank operator for all $N$. We will show that $S_{\epsilon,N}$ converges to $S_\epsilon$ in operator norm, i.e. $S_{\epsilon,N}f \to S_\epsilon f$ uniformly for all $f$ in the unit ball of $\Lpp(Q_0)$. Fix such an $f$. Observe that 
\[S_\epsilon f - S_{\epsilon,N} f = \sum_{Q \in D_N^c \cap \calS} \epsilon_Q \dashint_Q f(y) dy \; \chi_Q.\]
By Lemma \ref{AssocNormEquiv} we may choose a $g\in \Lcpp(Q_0)$ with $\|g\|_{\cpp}\leq 1$ such that 
\[\left\|\sum_{Q \in D_N^c \cap \calS} \epsilon_Q \dashint_Q f(y) dy\; \chi_Q\right\|_\pp \leq 2 \int_{Q_0} \left(\sum_{Q \in D_n^c \cap \calS} \epsilon_Q \dashint_Q f(y) dy\; \chi_Q(x) \right) g(x) dx.\]
We next use Remark \ref{SparsenessEquiv} in the same way as in the proof of Theorem \ref{HaarBound}, but we split $\epsilon$ into one factor of $\epsilon_Q^{1-2\alpha}$ and two factors of $\epsilon^\alpha$ before using Lemma \ref{HolderIneq}. This gives 
\begin{align*}
\int_{Q_0} \sum_{D_N^c \cap \calS} \epsilon_Q \dashint_Q f(y) dy \; \chi_Q(x)  g(x) dx & \leq 2\sum_{Q\in D_n^c \cap \calS} \epsilon_Q \int_{E_Q} \dashint_Q f(y) dy \dashint_Q g(x) dx dz \\
	& \leq 2 \sup_{Q\in D_N^c} \epsilon^{1-2\alpha} \sum_{Q\in D_N^c\cap \calS} \int_{E_Q} M_{\epsilon^\alpha} f(z) M_{\epsilon^\alpha} g(z) dz\\
	& \leq 2 \sup_{Q\in D_N^c}  \epsilon^{1-2\alpha}\int_{Q_0} M_{\epsilon^\alpha} f(z) M_{\epsilon^\alpha} g(z) dz\\
	& \leq 4 \sup_{Q\in D_N^c}  \epsilon^{1-2\alpha} \|M_{\epsilon^\alpha} f\|_{\pp} \|M_{\epsilon^\alpha} g\|_\cpp\\
\end{align*}

Since $\pp \in \epsilon^\alpha LH_0(Q_0)$, by Lemma \ref{EpsLH0ConjExpEquiv} we have that $\cpp \in \epsilon^\alpha LH_0(Q_0)$. Thus, by Lemma \ref{LocalEpsMaxOpBdd}, we have that
\[ \|M_{\epsilon^\alpha} f\|_{\pp} \|M_{\epsilon^\alpha} g\|_\cpp \leq \|f\|_\pp \|g\|_\cpp \leq 1,\]
and so we need only show that $\sup_{Q\in \D_N^c} \epsilon_Q^{1-2\alpha} \to 0$ as $N\to \infty$. 

Choose $N_0$ such that $2^{N_0} = \ell(Q_0)$. Then for all $N\geq N_0$, there are no cubes $Q\in \D(Q_0)$ such that $\ell(Q) >2^N$. Hence $D_N^c = \{Q \in \D(Q_0) : \ell(Q)<2^{-N}\}$. Since we assume $\lim_{N\to \infty} \sup\{\epsilon_Q : \ell(Q) < 2^{-N}\} = 0$, we have that $\sup_{Q\in D_N^c} \epsilon_Q^{1-2\alpha} \to 0$ as $N \to \infty$. Thus, $S_{\epsilon, N} \to S_\epsilon$, and so $S_\epsilon$ is a limit of finite rank operators. Hence, $S_\epsilon$ is compact on $\Lpp(Q_0)$ (See \cite[p. 174]{ConwayFuncAnal}). Consequently, Theorem \ref{PointwiseSparseBound} gives us that the Haar multiplier $T_\epsilon$ is also compact on $\Lpp(Q_0)$. 

\end{proof}


\section{Examples}
\label{sec:Examples}

In this section, we give sufficient conditions on the collection $\epsilon$ so that a specific exponent function $\pp$ is not locally log-Holder continuous, but is in $\epsilon LH_0(\R)$. Let $0<a<1$ and define 
\begin{align}
p(x) = \begin{cases} 2 & x \leq 0\\
						2+(\log_2 \frac{2}{x})^{-a} & 0 < x < 1\\
						3 & x \geq 1\end{cases}.
\end{align}

This exponent function is not in $LH_0(\R)$, as proved in \cite[Section 4.4]{VLS}. We next define the conditions on $\epsilon$ for which $\pp \in \epsilon LH_0(\R)$. Let $C \geq 1$. For each $n \in \Z$ with $n\geq 0$, let $Q_n = [0,2^{-n})$ and $Q_n' = [2^{-n-1}, 2^{-n})$. Let $\epsilon = \{\epsilon_Q\}_{Q\in \D}$ be any collection satisfying
\begin{align}
 \epsilon_{Q_n}&  \leq 2^{-n}C^{(n+1)^a} \text{ for all } n \geq 0, \label{OriginCubes}\\
\epsilon_{Q_n'} &= \epsilon_{Q_n} \text{ for all } n \geq 0 ,\label{NotOriginCubes}\\
\epsilon_Q & =C \text{ for all other cubes}\; Q
\end{align}

For such a collection $\epsilon$, we have $\pp \in \epsilon LH_0(\R)$. Let $n\geq 0$. First observe that since $(\log_2 (2/x))^{-a}$ is an increasing function, it attains its infimum at the left endpoint, and its supremum an the right endpoint of any cube. Consequently, we have
\begin{align}
p_-(Q_n) - p_+(Q_n)& = -(n+1)^{-a} \label{Qn}\\
p_-(Q_n') - p_+(Q_n') &= (n+2)^{-a} - (n+1)^{-a} \label{Qn'}
\end{align}

For the cube $Q_n$, by inequalities \eqref{Qn} and \eqref{OriginCubes}, we have
\begin{align*}
\left( \frac{|Q_n|}{\epsilon_{Q_n}}\right)^{p_-(Q_n) - p_+(Q_n)} = \left( \frac{2^{-n}}{\epsilon_{Q_n}} \right)^{-(n+1)^{-a}} = (2^n \epsilon_{Q_n})^{(n+1)^{-a}}
	\leq (C^{(n+1)^a})^{(n+1)^{-a}} = C.
\end{align*}

For the cube $Q_n'$, by \eqref{NotOriginCubes} and inequality \eqref{Qn'}, we have
\begin{multline*}
\left( \frac{|Q_n'|}{\epsilon_{Q_n'}} \right)^{p_-(Q_n')-p_+(Q_n')} = \left(\frac{2^{-n-1}}{\epsilon_{Q_n}}\right)^{(n+2)^{-a} - (n+1)^{-a}} = ( 2^{n+1}\epsilon_{Q_n})^{(n+1)^{-a}-(n+2)^{-a}}\\
	\leq (2C^{(n+1)^a})^{(n+1)^{-a} - (n+2)^{-a}} = 2^{(n+1)^{-a} - (n+2)^{-a}} C^{1-\left(\frac{n+2}{n+1}\right)^{-a}}
\end{multline*}

This expression is bounded, since both $2^{(n+1)^{-a}-(n+2)^{-a}}$ and $C^{1-\left( \frac{n+2}{n+1}\right)^{-a}}$ are decreasing and converge to $1$ as $n\to \infty$. Consequently, we obtain the upper bound at $n=0$, which is $(2C)^{1-2^{-a}}$. 

Now consider cubes of the form $Q=[0,2^k)$ for any $k \in \N$. Then
\begin{align*}
\left(\frac{|Q|}{\epsilon_Q}\right)^{p_-(Q)-p_+(Q)}  = \left(\frac{2^k}{C}\right)^{2-3} = \frac{C}{2^k} \leq \frac{C}{2}.
\end{align*}

Lastly, consider any cube $Q$ that is neither $Q_n$ nor $Q_n'$ for any $n \geq 0$ and is not of the form $[0,2^k)$ for $k \in \N$. These cubes do not intersect $[0,1)$ and so $\pp$ is constant on these cubes. Hence $p_-(Q) -p_+(Q)=0$ and so 
\[\left( \frac{|Q|}{\epsilon_Q}\right)^{p_-(Q)-p_+(Q)} = 1.\]
Thus, $\pp\in \epsilon LH_0(\R)$.

\bibliography{Bibliography.bib}{}
\bibliographystyle{plain}

\end{document}